\documentclass[11pt,a4paper]{amsart}

\newcommand{\PMP}{Pontryagin's Maximum Principle}
\newcommand{\R}[1]{\mathbb{R}^{#1}}
\newcommand{\N}[1]{\mathbb{N}^{#1}}
\newcommand{\innerprod}[2]{\ensuremath{\left\langle {#1} , {#2} \right\rangle}}
\newcommand{\lieprod}[2]{\ensuremath{\left[ {#1} , {#2} \right]}}
\newcommand{\ad}[3]{\ensuremath{\operatorname{ad}_{#1}^{#2}{#3}}}
\newcommand{\sign}[1]{\ensuremath{\operatorname{sign} \left( {#1} \right)}}
\newcommand{\mes}[1]{\ensuremath{\operatorname{Mes} \left( {#1} \right)}}
\newtheorem{teo}{Theorem}[section]
\newtheorem{lem}[teo]{Lemma}
\newtheorem{prop}[teo]{Proposition}
\newtheorem{cor}[teo]{Corolary}

\begin{document}
\title{The order of optimal control problems}
\author{Eduardo Oda}
\address{IMEUSP - Instituto de Matem\'{a}tica e Estat\'{\i}stica da Universidade de S\~{a}o Paulo}
\email{oda@ime.usp.br}
\author{Pedro Aladar Tonelli}
\address{IMEUSP - Instituto de Matem\'{a}tica e Estat\'{\i}stica da Universidade de S\~{a}o Paulo}
\email{tonelli@ime.usp.br}

\begin{abstract}
  The \PMP{} allows, in most cases, the design of optimal controls of affine nonlinear control systems by considering the sign of a smooth function. There are cases, although, where this function vanishes on a whole time interval and the \PMP{} alone does not give enough information to design the control. In these cases one considers the time derivatives of this function until a $k$-order derivative that explicitly depends on the control variable. The number $q=k/2$ is called \emph{problem order} and it is the same to all the extremals. The \emph{local order} is a related concept used in literature, but depending on each particular extremal. The confusion between these two concepts led to misunderstandings in past works (see \cite{LEWIS:1980}), where the \emph{problem order} was assumed to be an integer number. In this work we prove that this is true if the control system has a single input but, in general, it is not true if the control system has a multiple input.
\end{abstract}

\maketitle

\section{Introduction}
  Consider the following optimal affine control problem (P):
  \begin{align*}
    \mbox{minimize } & \int_{0}^{T_f(u)} f_0(x)+\sum_{i=0}^{m}g_{0i}(x)u_i \; dt \\
    \mbox{subject to } & \begin{cases}
        \dot{x} = f(x)+\sum_{i=0}^{m}g_i(x)u_i \\
        u=(u_1,\dots,u_m):[0,T_f(u)]\rightarrow\R{m} \mbox{ such that }\\
        |u_i(t)|\leq K(t),\; \forall t\in [0,T_f(u)]&, i=1, \dots, m \\
        x(0) = A \\
        x(T_f(u)) = B
      \end{cases}
  \end{align*}
  where:
  \begin{enumerate}
    \item $x=(x_1,\dots,x_n)\in\R{n}$ and $u=(u_1,\dots,u_m)\in\R{m}$
    \item $f$, $g_i$, $i=1, \dots, m$, are analytic vector fields in $\R{n}$
    \item $f_0$, $g_{0i}$, $i=1, \dots, m$, are analytic maps from $\R{n}$ to $\R{}$
    \item $K$ is analytic and strictly positive
    \item $u_i \in \mes{\R{}}$, $i=1, \dots, m$
  \end{enumerate}
  
  From the \PMP{}, if one defines the Hamiltonian function:
  \begin{align*}
    H_\lambda:T^*\R{n} \times \R{m} &\longrightarrow \R{} \\
    (x,p,u) &\longmapsto \innerprod{p}{f(x)+\sum_{i=0}^{m}g_i(x)u_i} - \lambda \left( f_0(x)+\sum_{i=0}^{m}g_{0i}(x)u_i \right)
  \end{align*}
  where $(x,p) \in T_x\R{n}$ and $\lambda \in \{0,1\}$, then each optimal trajectory $(\bar{x},\bar{u}):[0,\bar{T}] \rightarrow \R{n} \times \R{m}$ has a lift to the cotangent space such that $H_\lambda(\bar{x}(\bar{T}),\bar{p}(\bar{T}),\bar{u}(\bar{T}))=0$ and
  \begin{equation*}
    \mbox{(Adj) }
      \begin{cases}
      \frac{d\bar{x}}{dt}(t)=\frac{\partial H_\lambda}{\partial p}(\bar{x}(t),\bar{p}(t),\bar{u}(t)) \\
      \frac{d\bar{p}}{dt}(t)=-\frac{\partial H_\lambda}{\partial x}(\bar{x}(t),\bar{p}(t),\bar{u}(t)) \\
      H_\lambda(\bar{x}(t),\bar{p}(t),\bar{u}(t))=\sup\left\{ H_\lambda(\bar{x}(t),\bar{p}(t),v) | v \in U \right\}
      \end{cases}       
  \end{equation*} 
  for almost all $t \in [0,\bar{T}]$ and $\lambda \in \{0,1\}$. Also $(\lambda,\bar{p}(t))\neq 0$, for almost all $t\in [0,\bar{T}]$. The solutions of (Adj) are called extremals and might not be an optimal solution of the original problem. The new variable $p$ is known as the \emph{adjoint variable}.

  Setting
  \begin{align*}
    \bar{f}&=(f_0,f) & \bar{g}_i&=(g_{0i},g_i) \\
    \bar{x}&=(x_0,x) & \bar{p}&=(\lambda,p)
  \end{align*}
  where $x_0$ satisfies:
  \begin{equation*}
    \dot{x}_0=f_0(x)+\sum_{i=0}^{m}g_{0i}(x)u_i,
  \end{equation*}
  the Hamiltonian function become $H_\lambda=\innerprod{\bar{p}}{\bar{f}} + \sum_{i=i}^{m} \innerprod{\bar{p}}{\bar{g}_i}u_i$ and one has:
  \begin{equation} \label{e:adj_simples}
    \begin{aligned}
      \dot{\bar{x}}&=\frac{\partial H_\lambda}{\partial \bar{p}} = \bar{f}+\sum_{i=0}^{m}\bar{g}_i(x)u_i\\
      \dot{\bar{p}}&=-\frac{\partial H_\lambda}{\partial \bar{x}} = -\bar{p}\frac{\partial f}{\partial x} - \sum_{i=1}^{m} \bar{p}\frac{\partial g_i}{\partial x}u_i.
    \end{aligned}
  \end{equation}
  
  Henceforth, the simplified notation $f, g_i, x$ and $p$ stands for $\bar{f}, \bar{g}_i, \bar{x}$, respectively. Note that the dimension of the new problem is $n+1$.
  
  Since the Hamiltonian $H_\lambda$ is linear in $u$, one has by the \PMP{} that $u_i(t)=\sign{\innerprod{p(t)}{g_i(x(t))}} K(t)$ on the intervals where $\innerprod{p(t)}{g_i(x(t))}$ is non zero almost everywhere. These are known as the \emph{nonsingular intervals}. In this case the control is said nonsingular on these intervals.
    
  Analogously, the intervals where $\innerprod{p(t)}{g_i(x(t))}$ vanishes almost everywhere are known as \emph{singular intervals} and the controls are called singular on these intervals. Along a singular interval the control can not be designed by the \PMP{}, but one could consider the time derivatives of $\innerprod{p(t)}{g_i(x(t))}$.

  Indeed, in a singular interval, the time derivatives of the column vector $\left[\innerprod{p}{g_i}\right]_{i=1,\dots,m}$ could be evaluated until a relation that depends on $u$ explicitly is obtained. In other words, for $l \in \N{}$,  one has the $m \times m$ matrix:
  \begin{equation*}
    B_l = \frac{\partial}{\partial u} \left( \frac{d^l}{dt^l} \left[\innerprod{p}{g_i}\right]_{i=1,\dots,m} \right).
  \end{equation*}
  
  Consider the first one that is not identically zero, say the $k$-th derivative $B_k$. Then one has:
  \begin{equation*}
    0 = \frac{d^k}{dt^k} \left[\innerprod{p}{g_i}\right]_{i=1,\dots,m} = A_k(x,p) + B_k(x,p)u.
  \end{equation*}
  
  If $B_{k}$ is nonsingular, all the controls can be evaluated. Otherwise, it is necessary to find some controls, reducing the problem and restarting the procedure. It's now clear the central role of the functions $\innerprod{p(t)}{g_i(x(t))}$ and its derivatives on the design of optimal controls.
  
  The number $q=k/2$ is called \emph{problem order} or \emph{intrinsic order}. Note that even if $B_k$ is not identically zero it can became singular, or even identically zero, along an specific extremal. This fact lead us to another concept of order, know as \emph{arc order} or \emph{local order}. There were a lot of confusion around these concepts of order, first noticed by Lewis in 1980 \cite{LEWIS:1980}. There is one issue, although, that was not mentioned by Lewis, which will be treated in this paper.
  
  When Robbins, in 1967 \cite{ROBBINS:1967}, has enunciated and proved the Generalized Legendre-Clebsch (GLC) Condition, he was not aware of these order concepts, so he did not make clear which order he was considering. For that reason, some authors, including Lewis, have thought that the GLC Condition was valid for the problem order. Unfortunately it's not true, since the problem order can be a fraction, i.e., $k$ can be odd. Consider, for instance, the following optimal control problem:
  \begin{align*}
      \dot{x} &= v_1 \cos \theta + v_2 \sin \theta &
      \dot{y} &= v_2 \cos \theta - v_1 \sin \theta &
      \dot{\theta} &= \Omega \\
      \dot{v}_1 &= u_1 &
      \dot{v}_2 &= u_2 &
      \dot{\Omega} &= u_3.
  \end{align*}
  
  No matter what functional one wants to optimize, the intrinsic order of this problem is always $\tfrac{3}{2}$. However, if one wants to minimize $\int_O^T x^2+y^2+\theta^2 dt$, it's easy to see that, been at the origin, stay at the origin with a fully singular trajectory is the optimal solution. Fortunately, in this case the matrix $B_k$ is identically zero, so this example does not contradict the GLC Condition, if one considers it along the trajectory, in other words, if we consider the GLC Condition with the local order.
  
  The problem order fail to be an integer mainly because there are several inputs in this example. In the case one has just one input, the problem order is always a positive integer, as will be proved soon.
  
\section{Evaluating the problem order}
  To evaluate the problem order a simple known lemma is necessary.
  
  \begin{lem} \label{l:derivada_de_<p,h>}
    Let $h$ be a smooth vector field. Then along an extremal we have:
    \begin{equation*}
      \frac{d}{dt}\innerprod{p}{h}=\innerprod{p}{\lieprod{f}{h}+\sum_{i=1}^{m}u_i\lieprod{g_i}{h}}.
    \end{equation*}
  \end{lem}
  \begin{proof}
    We know that
    \begin{equation*}
      \frac{d}{dt}\innerprod{p}{h}=\innerprod{\dot{p}}{h} + \innerprod{p}{\frac{\partial h}{\partial x}\dot{x}}.
    \end{equation*}

    From the equations \eqref{e:adj_simples} and the inner product linearity, the first portion of the sum above is:
    \begin{equation*}
      \innerprod{\dot{p}}{h}= \innerprod{p}{-\frac{\partial f}{\partial x}h} + \sum_{i=1}^{m}\innerprod{p}{-\frac{\partial g_i}{\partial x}}u_i
    \end{equation*}
    and the second one is:
    \begin{equation*}
      \innerprod{p}{\frac{\partial h}{\partial x}\dot{x}}=\innerprod{p}{\frac{\partial h}{\partial x}f}+\sum_{i=1}^{m}\innerprod{p}{\frac{\partial h}{\partial x}g_i}u_i
    \end{equation*}
    from which follows the result.
  \end{proof}

  One uses this lemma to evaluate the derivatives of $\phi = \left[\innerprod{p}{g_i}\right]_{i=1,\dots,m}$:
    
  \begin{equation}
    \phi^{(1)} =
      \left[\begin{array}{c}
        \frac{d}{dt}\innerprod{p}{g_1} \\
        \vdots \\
        \frac{d}{dt}\innerprod{p}{g_m} 
      \end{array}\right]
      =
      \left[\begin{array}{c}
        \innerprod{p}{ \ad{f}{}{g_1} + \sum_{i=1}^{m} \lieprod{g_i}{g_1}u_i } \\
        \vdots \\
        \innerprod{p}{ \ad{f}{}{g_m} + \sum_{i=1}^{m} \lieprod{g_i}{g_m}u_i }
      \end{array}\right].
  \end{equation}
    
  If the first derivative do not depend on $u$, in other words, if $\lieprod{g_i}{g_j}=0$, $\forall i,j$, then the second derivative is evaluated:
  \begin{equation}
    \phi^{(2)} =
      \left[\begin{array}{c}
        \frac{d}{dt}\innerprod{p}{ \ad{f}{}{g_1} } \\
        \vdots \\
        \frac{d}{dt}\innerprod{p}{ \ad{f}{}{g_m} }
      \end{array}\right]
      =
      \left[\begin{array}{c}
        \innerprod{p}{ \ad{f}{2}{g_1} + \sum_{i=1}^{m} \lieprod{g_i}{\ad{f}{}{g_1}}u_i } \\
        \vdots \\
        \innerprod{p}{ \ad{f}{2}{g_m} + \sum_{i=1}^{m} \lieprod{g_i}{\ad{f}{}{g_m}}u_i }
      \end{array}\right].
  \end{equation}
    
  Proceeding in this way until one finds an expression which explicitly depends on $u$ one gets:
  \begin{equation}
    \phi^{(k)} =
      \left[\begin{array}{c}
        \innerprod{p}{ \ad{f}{k}{g_1} + \sum_{i=1}^{m} \lieprod{g_i}{\ad{f}{k-1}{g_1}}u_i } \\
        \vdots \\
        \innerprod{p}{ \ad{f}{k}{g_m} + \sum_{i=1}^{m} \lieprod{g_i}{\ad{f}{k-1}{g_m}}u_i }
      \end{array}\right].
  \end{equation}
  
  It is possible to rewrite this expression in the form $\phi^{(k)} = A_k + B_ku$ where:
  \begin{equation*}
    A_k=\left[\begin{array}{c}
      \innerprod{p}{\ad{f}{k}{g_1}} \\
      \innerprod{p}{\ad{f}{k}{g_2}} \\
      \vdots \\
      \innerprod{p}{\ad{f}{k}{g_m}}
    \end{array}\right]
  \end{equation*}
  \begin{equation*}
    B_k=\left[\begin{array}{cccc}
      \innerprod{p}{\lieprod{g_1}{\ad{f}{k-1}{g_1}}} & \innerprod{p}{\lieprod{g_2}{\ad{f}{k-1}{g_1}}} & \cdots & \innerprod{p}{\lieprod{g_m}{\ad{f}{k-1}{g_1}}} \\
      \innerprod{p}{\lieprod{g_1}{\ad{f}{k-1}{g_2}}} & \innerprod{p}{\lieprod{g_2}{\ad{f}{k-1}{g_2}}} & \cdots & \innerprod{p}{\lieprod{g_m}{\ad{f}{k-1}{g_2}}} \\
      \vdots & \vdots & \ddots & \vdots \\
      \innerprod{p}{\lieprod{g_1}{\ad{f}{k-1}{g_m}}} & \innerprod{p}{\lieprod{g_2}{\ad{f}{k-1}{g_m}}} & \cdots & \innerprod{p}{\lieprod{g_m}{\ad{f}{k-1}{g_m}}}
    \end{array}\right].
  \end{equation*}

\section{Main result}
  In this section one considers nonlinear affine control systems with single input, $u\in\R{}$, and it is proved that the intrinsic order of optimal control problems of these systems is always a positive integer.
  
  \begin{teo} \label{p:q inteiro}
    Given an optimal control problem in the form (P) with single input, i.e., $m=1$, then its intrinsic order is a positive integer.
  \end{teo}

  To simplify the proof of this theorem, some useful results are proved first.
  
  \begin{prop} \label{l:passa j}
    If there is $k\in \N{}$ such that $\lieprod{g}{\ad{f}{i}{g}}=0$ for all $x$, where $0 \leq i \leq k-1$ then for all $j \in \N{}$, $1 \leq j \leq k$, one has:
    \begin{equation*}
      \lieprod{\ad{f}{j}{g}}{\ad{f}{l}{g}}=-\lieprod{\ad{f}{j-1}{g}}{\ad{f}{l+1}{g}}
    \end{equation*}
    for all $x$, with $0\leq l \leq k-(j+1)$.
  \end{prop}
  \begin{proof} 
    The proof uses induction on $j$. One knows that if $0 \leq l \leq k-1$ then $\lieprod{g}{\ad{f}{l}{g}}=0$, therefore $\lieprod{f}{\lieprod{g}{\ad{f}{l}{g}}}=0$. From the Jacobi Identity follows:
    \begin{equation*} 
      \lieprod{\ad{f}{1}{g}}{\ad{f}{l}{g}}=-\lieprod{g}{\ad{f}{l+1}{g}}
    \end{equation*}
    thus the proposition is true when $j=1$.
    
    Now, suppose that there is a $j_0<k$ such that the thesis is valid for all $j \leq j_0$, i.e.,
    \begin{equation*}
      \lieprod{\ad{f}{j}{g}}{\ad{f}{l}{g}}=-\lieprod{\ad{f}{j-1}{g}}{\ad{f}{l+1}{g}}
    \end{equation*}
    for all $x$, with $0\leq l \leq k-(j+1)$.
    
    Then, for all $j<j_0$, if $0\leq l \leq k-(j+1)$, one shows by induction in $m$ that:
    \begin{equation} \label{e:inducao em m}
      \lieprod{\ad{f}{m}{g}}{\ad{f}{l+j-m}{g}}=0
    \end{equation}
    for all $x$, with $0 \leq m \leq j$. Indeed, if $m=0$, we have $\lieprod{g}{\ad{f}{l+j}{g}}$, which is null by the proposition hypothesis since $l+j\leq k-1$. If there is a $m_0<j$ such that the relation is valid for all $m \leq m_0$ then, since we have the following inequalities:
    \begin{align*}
      m+1 &\leq j \\
      l+j-(m+1) & \leq k-(m+1),
    \end{align*}
    then, by the induction hypothesis on $m$, one has:
    \begin{equation*} 
      0=\lieprod{\ad{f}{m}{g}}{\ad{f}{l+j-m}{g}},
    \end{equation*}
    but, by the induction hypothesis on $j$, one gets:
    \begin{equation*} 
      \lieprod{\ad{f}{m+1}{g}}{\ad{f}{l+j-(m+1)}{g}}=-\lieprod{\ad{f}{m}{g}}{\ad{f}{l+j-m}{g}}
    \end{equation*}
    which is precisely the relation for $m+1$, ending the induction on $m$:

    So, taking $m=j$ in equation \eqref{e:inducao em m}, it follows that $\lieprod{\ad{f}{j}{g}}{\ad{f}{l}{g}}=0$, and this implies that $\lieprod{f}{\lieprod{\ad{f}{j}{g}}{\ad{f}{l}{g}}}=\lieprod{f}{0}=0$. Again, by the Jacobi Identity,
    \begin{equation*}
      \lieprod{\ad{f}{j+1}{g}}{\ad{f}{l}{g}}=-\lieprod{\ad{f}{(j+1)-1}{g}}{\ad{f}{l+1}{g}},
    \end{equation*}
    ending the induction on $j$.
  \end{proof}

  \begin{cor}
    With the same hypothesis of Proposition \ref{l:passa j}, if $0 \leq j \leq k$, then
    \begin{equation*}
      \lieprod{g}{\ad{f}{k}{g}}=(-1)^j\lieprod{\ad{f}{j}{g}}{\ad{f}{k-j}{g}}.
    \end{equation*}
  \end{cor}
  \begin{proof}
    Clearly, the assertion is true for $j=0$. Suppose that it's also true for all $j<k$. One has already these inequalities,
    \begin{gather*}
      1 \leq j+1 \leq k \\
      k-j-1 = k-(j+1) \leq k-(j+1)
    \end{gather*}
    and use the Proposition \ref{l:passa j}. Therefore
    \begin{equation*} 
      \lieprod{g}{\ad{f}{k}{g}}=(-1)^j\lieprod{\ad{f}{j}{g}}{\ad{f}{k-j}{g}}=(-1)^{j+1}\lieprod{\ad{f}{j+1}{g}}{\ad{f}{k-(j+1)}{g}}.
    \end{equation*}
  \end{proof}

  \begin{cor} \label{c:k_par_[g,ad_f^kg]=0}
    With the same hypothesis of Proposition \ref{l:passa j}, if $k$ is even then $\lieprod{g}{\ad{f}{k}{g}}=0$.
  \end{cor}
  \begin{proof}
    It's straightforward since $\lieprod{g}{\ad{f}{k}{g}}=(-1)^{\frac{k}{2}}\lieprod{\ad{f}{\frac{k}{2}}{g}}{\ad{f}{\frac{k}{2}}{g}}$.
  \end{proof}

  Now it is very simple to prove the Theorem \ref{p:q inteiro}.

  \begin{proof}[Proof of Theorem \ref{p:q inteiro}]
    Suppose that all the derivatives up to $k$ of $\innerprod{p}{g}$ does not depend explicitly of $u$, i.e., the control appears with a identically zero coefficient. So, from Lemma \ref{l:derivada_de_<p,h>}, it is known that $\innerprod{p}{\lieprod{g}{\ad{f}{i}{g}}}=0$, for all $i=0,\ldots,k-1$. 
    
    Since $p$ is arbitrary, one concludes also that $\lieprod{g}{\ad{f}{i}{g}}=0$, for all $i=0,\ldots,k-1$.
    
    Note that $k$ is not necessarily the problem order, but $k$ is less than or equal it.
    
    If $k$ is even, one can  use the Corollary \ref{c:k_par_[g,ad_f^kg]=0} to conclude that $\lieprod{g}{\ad{f}{k}{g}}=0$, and from that $\innerprod{p}{\lieprod{g}{\ad{f}{k}{g}}}=0$.
    
    Using again the Lemma \ref{l:derivada_de_<p,h>}, it is easy to see that $\innerprod{p}{\lieprod{g}{\ad{f}{k}{g}}}$ is precisely the coefficient of $u$ in the derivative of order $k+1$ of $\innerprod{p}{g}$. Therefore the first derivative of $\innerprod{p}{g}$, that explicitly depends of $u$ can not be odd. Therefore the problem order is always even.
  \end{proof}
  
  The problem order, unlike the local order, can be easily evaluated. Thereby, it allows one to develop tools that can be directly applied to solve practical problems, like the design of optimal controls on a singular interval, and to answer theoretical questions, like the qualitative behavior of optimal controls at junctions of singular and nonsingular intervals. This paper has given a proof that the problem order is certainly a positive integer only if the system has a single input.

\bibliography{problem_order}
\bibliographystyle{plain}
\nocite{POWERS:1980}
\nocite{ODA:2008}
\nocite{CHYBA-LEONARD-SONTAG:2003}
\nocite{PONTRYAGIN:1962}
\nocite{ISIDORI:1989}
\nocite{BELL:1993}
\nocite{ODIA-BELL:2012}
\nocite{MEESOMBOON-BELL:2002}

\end{document}